\documentclass[10pt]{article}
\usepackage{amssymb}

\usepackage[a4paper,left=30mm,top=40mm,body={155mm,230mm}]{geometry}
\usepackage{stmaryrd}
\usepackage{bbding,bbm}
\usepackage{enumerate}

\usepackage{fancybox}
\usepackage{algorithm}
\usepackage{algorithmic}

\usepackage{fancyhdr,color,graphicx,amsmath}

\usepackage{multirow}

\usepackage{enumerate}

\usepackage{graphicx}
\usepackage{color}
\usepackage{ifpdf}

\usepackage{amsmath}

\usepackage{amsthm}
\allowdisplaybreaks[4]
\usepackage{mathrsfs}

\usepackage{xcolor}

\newtheorem{theorem}{\qquad \bf Theorem}[section]

\newtheorem{lemma}[theorem]{\qquad\bf Lemma}

\newtheorem{assumption}[theorem]{\qquad \bf Assumption}

\newcommand{\be}{\begin{equation}}
\newcommand{\ee}{\end{equation}}
\newcommand{\ba}{\begin{array}}
\newcommand{\ea}{\end{array}}
\newcommand{\bea}{\begin{eqnarray}}
\newcommand{\eea}{\end{eqnarray}}

\begin{document}

\title{An interior point method for nonlinear optimization with a quasi-tangential subproblem}
%\thanks{The work was supported by Chinese NSF grant 60873116.}}
%\author{Songqiang Qiu, Zhongwen Chen}
\author{SONGQIANG \ QIU\thanks{School of Sciences, China University of Mining and Technology,
Xuzhou, 221116, PR China}\quad AND\ ZHONGWEN\ CHEN\thanks{School of Mathematics Science, Soochow
University, Suzhou, 215006, PR China (zwchen@suda.edu.cn).}}
\date{}
\maketitle

\begin{abstract}
In this paper, we proposed an interior point method for constrained optimization, which is characterized by the using of quasi-tangential subproblem. This algorithm follows the main ideas of primal dual interior point methods and Byrd-Omojokun's step decomposition strategy. The quasi-tangential subproblem is obtained by penalizing the null space constraint in the tangential subproblem. The resulted quasi-tangential step is not strictly lying in the null space of the gradients of constraints. We also use a line search trust-funnel-like strategy, instead of penalty function or filter technology, to globalize the method. Global convergence results were obtained under standard assumptions.
\end{abstract}

\textbf{keywords:}
Constrained optimization;
Interior point method; quasi-tangential subproblem; Trust-funnel-like method;
 Global convergence

\section{Introduction}
In this paper, we describe and analyze an interior point method for nonlinear constrained optimization
problem
\begin{equation}\label{a1}
\begin{split}
\min&\  f(x)\\
\textrm{s.t.}&\ c(x)=0,\\
&\ x\geq 0,
\end{split}
\end{equation}
where $f:R^n\rightarrow R$, $c:R^n\rightarrow R^m$ are smooth functions. Problems with general nonlinear inequality constraints can be
equivalently reformulated into the above form by using slack variables.

Interior point methods provide a class of tools for the treatment of inequality constraints. They have been intensively studied in the last three
decades; see \cite{ByrdHN1999,CurtiSW2010,El-BaTTZ1996,Gondz2012,UlbriUV2004} and etc.
Efficient implementations based on interior point methods for solving linear and nonlinear programming have emerged; see \cite{Vande1999,WaB2006} and etc.
The classical interior point strategy for \eqref{a1}
obtains a solution by approximately solving
a series of barrier problem of the form
\begin{equation}\label{b1}
\begin{split}
\min&\ \varphi^\mu(x)\stackrel{\textrm{def}}{=}f(x)-\mu\displaystyle\sum_{i=1}^n\ln x^{(i)}\\
\textrm{s.t.}&\ c(x)=0,
\end{split}
\end{equation}
with $\mu\searrow 0$. So the solution of the above equality constrained barrier subproblem that generates steps is of crucial importance in interior point methods.
For this solution, some algorithms use perturbed optimality conditions \cite{ForsgG1998,YamasY1996} while  some involve SQP (trust region) mechanisms.
SQP (trust region) based interior point methods, see \cite{ByrdHN1999,ByrdGN2000,CurtiSW2010,CurtiHSW2012,NocedW2014} and etc, have been proven to be robust and efficient. Of all the methods of this kind, the step-decomposition approaches,  which splits the total step into a normal
step and a tangential step, and which integrates ideas of interior point methods and
Byrd-Omojokun's trust region idea \cite{Byrd1987,Omojo1989}, have been proven practical. Two nice features of this strategy is the consistent subproblems and the capacity of infeasibility detection \cite{NocedW2014}.

Our approach follows this algorithmic frame with a major character that it employes a quasi-tangential subproblem which generates a step not strictly lying in the null space of $\nabla c_k^T$.
The use of quasi-tangential subproblems in Byrd-Omojokun-like methods is not entirely new. Some methods with inexact step computation, in interior points or trust region (SQP) scheme, adopt similar ideas. Curtis et al \cite{CurtiSW2010} used an inexact Newton technique in their interior point methods where an inexact tangential step
is generated from an inexact Newton equation for the tangential subproblem. Heinkenschloss and Ridzal \cite{raey} obtained inexact tangential steps by computing approximate projections of vectors onto the null space of the constraints Jacobian. Gould and Toint \cite{GouldT2010} specified conditions that
 an (inexact) tangential step should satisfy but did not give computational details. In our method, the key point of the quasi-tangential subproblem is to convert the tangential subproblem into an unconstrained
quadratic programming by penalizing the null space constraints. An inexact tangential step satisfying some necessary conditions is obtained if choosing sufficiently small penalty factor. This strategy circumvents the computation of a base for the null space, which is quite important for the solution
of tangential subproblem in Byrd-Omojokun-like algorithms.

Another character
of our algorithm is that we use a trust-funnel-like strategy to balance the improvements on feasibility and optimality. Trust funnel method was introduced by
Gould and Toint in \cite{GouldT2010}. And similar ideas can be found in \cite{BielsG2008,QiuC2012,QiuC2012a,ShenXP2010,XueSP2009} and etc. Our balance idea mainly
differs from these algorithms in the switch conditions between so called $f-$iteration and $h-$iteration.

The balance of this paper is organized as follows. In the next section, we describe the design of the algorithm in detail. In Section 3, we show that the
proposed algorithm is well-defined while the global convergence is shown in Section 4. In Section 5, we report preliminary numerical results. Finally, some
further remark is given in Section 6.

\emph{Notations:} We use $\|\cdot\|$ to denote the Euclidean norm $\|\cdot\|_2$. Subscript $k$ refers to iteration indices and superscript $(i)$
is the $i$th component of a vector.

\section{Algorithm Description}

We motivate the main algorithm in this section. Our method employs Armijo line search to get global convergence. Yet, we are partially motivated by trust region methods for \eqref{b1}.
So we give a brief overview of this kind of methods, in a primal version.
The sequential quadratic programming approach can be applied to (approximately) solving the barrier problem \eqref{b1}. At an iterate $x$, a step $d$ is a
displacement generated by solving
\begin{equation}\label{b2}
\begin{split}
\min&\ \nabla \varphi^\mu(x)^Td+\frac12d^TWd\\
\textrm{s.t.}&\ c(x)+\nabla c(x)^Td=0,
\end{split}
 \end{equation}
where $W=H+\mu X^{-2}$ with $H$ the Hessian of the Lagrangian of problem \eqref{b1} and $X=\textrm{diag} (x)$. A trust region constraint \begin{equation}\label{tr}\|d\|\leq\Delta\end{equation} is always introduced in \eqref{b2} so as to obtain global convergence and to allow for the case where
$W$ is not positive definite on the null space of $\nabla c(x)^T$.
 It is well known \cite{Vardi1985} that \eqref{pdqp} with \eqref{tr} can be
inconsistent when the trust region $\Delta$ is so small that even the shortest step $d$ that satisfies the constraints in \eqref{pdqp}
is excluded by the trust region.
One of the strategies to make the constraints consistent is the step decomposition method of Byrd \cite{Byrd1987}
and Omojokum \cite{Omojo1989} in which the total step of the algorithm is split into normal and tangential steps.
The normal step $v$ is a move toward the satisfaction of the the constraints of \eqref{pdqp}, and is defined as the solution
of the normal subproblem
\begin{equation}\label{b3}
\begin{split}
\min&\ \frac12\|c(x)+\nabla c(x)^Tv\|^2\\
\textrm{s.t.}&\ \|v\|\leq \xi\Delta,
\end{split}
\end{equation}
with $\xi\in(0,1)$.
The tangential step $t$ aims to reduce $\varphi^\mu(x)$ on the null space of $\nabla c(x)^T$, and is generated by solving the tangential subproblem
\begin{equation}\label{b4}
\begin{split}
\min_t &\ (\nabla \varphi^\mu(x)+Wv)^Tt+\frac12t^TWt\\
\textrm{s.t.}&\ \nabla c(x)^Tt=0,\\
& \|v+t\|\leq\Delta.
\end{split}
\end{equation}
Following the above ideas, however, in this paper, we wish to present an alternative to the null space constraint in tangential subproblem \eqref{b4}. This and the well-know equivalence between trust region methods and Levenberg-Marquardt methods provide the main motivations for our algorithm.

\subsection{The Primal-Dual Barrier Method}

The Karush-Kuhn-Tucker (KKT) conditions of the barrier problem \eqref{b1} cause the following nonlinear system
\begin{equation}\label{pkkt}
\begin{pmatrix}
\nabla f(x)+\nabla c(x)\lambda-z\\
-\mu X^{-1}e+z\\
c(x)
\end{pmatrix}=0
\end{equation}
with $\lambda\in R^m$, $0\leq z\in R^n$ the Lagrangian multipliers.
Multiplying the second row of \eqref{pkkt} by $X$, we obtain the system
\begin{equation}\label{pdkkt}
\begin{pmatrix}
\nabla f(x)+\nabla c(x)\lambda-z\\
Xz-\mu e\\
c(x)
\end{pmatrix}=0.
\end{equation}
This may be viewed as a perturbed KKT system for the original problem \eqref{a1}. The optimality error for the barrier problem here is defined, based on \eqref{pdkkt} as \cite{WaB2006}
\begin{equation*}
E_{\mu}(x,\lambda,z);=\left\{\frac{\|\nabla f(x)+\nabla c(x)\lambda-z\|}{s_d},\frac{\|Xz-\mu e\|}{s_c},\|c(x)\|\right\}
\end{equation*}
with scaling parameters $s_d,s_c\geq1$ defined as
$$
s_d=\max\left\{s_{\max},\frac{\|\lambda\|_1+\|z\|_1}{m+n}\right\}/s_{\max},\ \ s_c=\max\left\{s_{\max},\frac{\|z\|_1}{n}\right\}/s_{\max},
$$
where $s_{\max}>1$ fixed. Correspondingly, we use $E_0(x,\lambda,z)$ to measure the optimality error for the original problem \eqref{a1}. A
typical algorithmic framework of barrier methods for \eqref{a1} is as follows.

\framebox{\parbox{0.9\textwidth}{

A{\scriptsize LGORITHM} 1:  O{\scriptsize UT} L{\scriptsize OOP}
\begin{enumerate}[\bf Step\ 1]\setcounter{enumi}{-1}
\item Choose an initial value for the barrier parameter $\mu_0>0$, and select the parameter $\kappa_\epsilon>0$, and the stop tolerance $\epsilon_{tol}$.
Choose the starting point $x_0$, $\lambda_0$, $z_0$. Set $j:=0$.

\item If $E_0(x_j,\lambda_j,z_j)\leq\epsilon_{tol}$, stop.

\item Apply an SQP method, starting from $x_j$, to find an approximate solution $x_{j+1}$ for the barrier problem, with Lagrange multipliers
$\lambda_{j+1},$ $z_{j+1}$ satisfying $E_{\mu_j}(x_{j+1},\lambda_{j+1},z_{j+1})\leq\kappa_\epsilon\mu_j$.

\item Choose $\mu_{j+1}\in(0,\mu_j)$, set $j:=j+1$ and go to Step 1.
\end{enumerate}
}}

~

\emph{Remark:} To achieve fast local convergence algorithm, the barrier parameter $\mu$ needs to be updated carefully \cite{El-BaTTZ1996,YamasY1996}.
We will follow the approach suggested by Byrd, Liu and Nocedal \cite{byrd1997local} and will specify the details in Section 5.

The primary work of Algorithm 1 lies clearly in Step 2, where an approximate solution of \eqref{b1} is found.
Primal-dual interior point methods apply Newton's method to the perturbed KKT systems
\eqref{pdkkt} and modify step lengths so that the inequalities $(x,z)\geq 0$ are satisfied strictly. A primal-dual linear system
is given as
\begin{equation*}
\begin{pmatrix}
H&\nabla c(x)& -I\\
Z&0&X\\
\nabla c(x)^T&0&0
\end{pmatrix}=\begin{pmatrix}
d_x\\ d_{\lambda}\\ d_z
\end{pmatrix}=
-\begin{pmatrix}
\nabla f(x)+\nabla c(x)\lambda-z\\
Xz-\mu e\\
c(x)
\end{pmatrix},
\end{equation*}
 where we have defined
$
Z=\textrm{diag}(z_1,\cdots,z_n)$. Eliminating $d_z$ by
\begin{equation}\label{dz}
d_z=-z+\mu X^{-1}e-X^{-1}Zd_x,
\end{equation}
yields the iteration
\begin{equation*}
\begin{pmatrix}
H+X^{-1}Z&\nabla c(x)\\
\nabla c(x)^T&0
\end{pmatrix}\begin{pmatrix}d_x\\ \lambda^+
\end{pmatrix}=-\begin{pmatrix}
\nabla \varphi(x)\\ c(x)\end{pmatrix}.
\end{equation*}
It is easy to see that the step generated by this system coincides with the solution of the following primal-dual QP subproblem
\begin{equation}\label{pdqp}
\begin{split}
\min&\ \nabla \varphi^\mu(x)^Td+\frac12d^T\tilde Wd\\
\textrm{s.t.}&\ c(x)+\nabla c(x)^Td=0,
\end{split}
\end{equation}
where $\tilde W=H+X^{-1}Z.$ Steps computation of our algorithm is based on this mode.

\subsection{Trust-Funnel-Like Approach For The Barrier Problem}

We use a trust-funnel-like \cite{GouldT2010,QiuC2012} line search algorithm for the approximate solution of the barrier problem \eqref{b1} with $\mu=\mu_j$.
Given the iterate $x_k$ and the corresponding Lagrange multiplier $\lambda_k$, $z_k$, a trust-funnel-like method for constrained optimization pursuits the solution iterately in a progressively stricter trust-funnel defined by
 \begin{equation*}
 h(x)\leq h_k^{\max},
 \end{equation*}
 where
 $$
 h(x)=\|c(x)\|
 $$
 and $h_k^{\max}$ is a nonincreasing limit on infeasibility.

Note that \eqref{b3} is an trust region model for the nonlinear least square problem
\begin{equation*}
\min\ h(x)^2.
\end{equation*}
As a classical but still popular method for this problem, Levenberg-Marquardt method \cite{Leven1944,Marqu1963,Mor1978} computes a search direction
by solving the following linear system
\begin{equation}\label{b8}
(\nabla c_k\nabla c_k^T+\eta_k I)v=-\nabla c_kc_k,
\end{equation}
where $\eta_k$ is a positive parameter. Sometimes it is considered to be
the prognitor of the trust region approach for general unconstrained optimization \cite{WrighN1999}. The following lemma from Nocedal and Wright \cite{WrighN1999} shows connection between the solutions of \eqref{b3} and \eqref{b8}.
\begin{lemma}
The vector $v_k^{\textrm{LM}}$ is a solution of the trust region subproblem \eqref{b3} if and only if $v_k^{\textrm{LM}}$ is feasible and
there is a scalar $\rho>0$ such that
\begin{eqnarray*}
&(\nabla c_k\nabla c_k^T+\eta_k I)v_k^{\textrm{LM}}=-\nabla c_kc_k,\\
&\eta_k(\Delta-\|v_k^{\textrm{LM}}\|)=0.
\end{eqnarray*}
\end{lemma}
Levenberg-Marquardt method is globalized by a line search strategy, which is less costly in computation than trust region method. Furthermore, researchers
have shown that if $\eta_k=\|c_k\|^\delta$, $\delta\in(0,2]$,
fast local convergence can be achieved without nonsingularity assumption, see Yamashita and Fukushima\cite{yamashita2001rate}, Fan and Yuan \cite{FanY2005}, Zhang \cite{Zhang2003} and etc. These facts motivate us to compute the normal step by the following hybrid method. Hereby, the normal search direction is defined as
\begin{equation}\label{b9}
v_k=\begin{cases}\arg\min\ \|c_k+\nabla c_k^Tv\|^2+\|c_k\|^{\delta}\|v\|^2,& \textrm{if }\nabla c_k \textrm{ is rank deficient}, \\
\arg\min\|c_k+\nabla c_k^Tv\|^2,& \textrm{otherwise},\end{cases}
\end{equation}
where $\delta\in(0,2)$ is a fixed constant.

%The tangential constraint $\nabla c(x)^Tt=0$ prevents $t$ from jeopardizing the infeasibility reduction that normal step $v_k$ just obtains.
%To achieve similar goal, this constraint can be relaxed as
%\begin{equation}\label{b10}
%\|\nabla c(x)^Tt\|\leq\eta
%\end{equation}
%for some positive scalar $\eta$. Similar ideas have been proposed by
% Gould and Toint in \cite{GouldT2010}. Like their ideas, such relaxation aims to trade some decrease in linearized
%feasibility for a large improvement in objective value over a reasonable step. At first glance, such a modification increase the difficulty of solving
%tangential subproblem. But this side effect can be easily circumvented by adding an item $\frac1{2\nu}t^T\nabla c(x)\nabla c(x)^Tt$ to the quadratic
%objective function where $\nu>0$. Hereby, we introduce our inexact tangential subproblem
%\begin{equation}\label{b11}
%\min_t\ (\nabla \varphi^\mu(x)+Ws)^Tt+\frac12t^T(W+\frac1{\nu}\nabla c(x)\nabla c(x)^T+\zeta I)t,
%\end{equation}
%where $\zeta>0$ is a scalar such that $(W+\frac1{\nu}\nabla c(x)\nabla c(x)^T+\zeta I)$ is positive definite.
The effect of the tangential constraint $\nabla c(x)^Tt$ is to prevent $t$
from jeopardizing the infeasibility reduction that normal step $v_k$ just obtains. It has been observed that this effect can be achieve by only requiring
\begin{equation}\label{b10}
\|\nabla c_k^Tt\|\leq\xi_k
\end{equation}
for an appropriate positive scalar $\xi_k$ \cite{GouldT2010}. On the other hand, it is quite likely that the concession on linearized feasibility that \eqref{b10} makes
can lead to some considerable improvement in objective value over a reasonable step. However, replacing $\nabla c_k^Tt=0$ by \eqref{b10} will undoubtedly
increase the difficulty of solving tangential subproblem. But this difficulty can be easily circumvented by adding the item $\frac1{2\nu_k}t^T\nabla c_k\nabla c_k^Tt$ to the quadratic
objective function where $\nu_k>0$ acts as a penalty factor. Hereby, we get our quasi-tangential subproblem
\begin{equation}\label{b11}
\min_t\ (\nabla \varphi^\mu_k+\tilde W_kv_k)^Tt+\frac12t^T\left(\tilde W_k+\frac1{\nu_k}\nabla c_k\nabla c_k^T+\zeta_k I\right)t.
\end{equation}
where $\zeta_k\geq0$ is a scalar which is chosen such that $(\tilde W_k+\frac1{\nu_k}\nabla c_k\nabla c_k^T+\zeta_k I)$ is positive definite.
%\begin{equation}\label{pdef}
%d^T(\tilde W_k+\frac1{\nu_k}\nabla c_k\nabla c_k^T+\zeta_k I)d\geq b_1\|d\|^2
%\end{equation}
%for any $d\in r^n$ with a fixed $b_1>0$.

Now we specify the choices of $\nu_k$. Firstly, the parameter $\nu_k$ should be a positive scalar that ensures the positive
 definiteness of the symmetric matrix $\tilde W_k+\frac1{\nu_k}\nabla c_k\nabla c_k^T$. Secondly, $\nu_k$ should be small enough
  such that the corresponding solution $t_k$ of \eqref{b11} satisfies
either \begin{eqnarray}
&-(\nabla\varphi^{\mu_j}_k)^T(v_k+t_k)\geq \sigma_1h_k^{\sigma_2},\label{b16}\ \textrm{and}\\
&\|\nabla c_k^Tt_k\|\leq\kappa_1(h_k^{\max}-\|c_k+\nabla c_k^Tv_k\|),\label{b17}
\end{eqnarray}
where $\sigma_1$, $\sigma_2$ are positive constants and $\kappa_1\in(0,1)$,
 or
\begin{eqnarray}
&-(\nabla\varphi^{\mu_j}_k)^T(v_k+t_k)< \sigma_1h_k^{\sigma_2},\label{b18}\ \textrm{and}\\
&\|\nabla c_k^Tt_k\|\leq\kappa_2(h_k-\|c_k+\nabla c_k^Tv_k\|),\label{b19}
\end{eqnarray}
where $\kappa_2\in(0,1)$. The positive definiteness of $\tilde W_k+\frac1{\nu_k}\nabla c_k\nabla c_k^T$ is guaranteed for a small enough
 parameter $\nu_k$ if $H_k$
is positive definite on the null space of $\nabla c_k^T$.

\begin{lemma}\label{thmnu}
Suppose that $H_k$ is positive definite on the null space of $\nabla c(x)$, and that $\nabla c(x)$ has full rank. Then there is a
threshold value $\bar \nu$ such that for all $\nu\in(0,\bar\nu]$, $\tilde W_k+\frac1{\nu_k}\nabla c_k\nabla c_k^T$ is positive definite.
\end{lemma}
\begin{proof}
Since $X_k^{-1}Z_k$ is strictly positive definite, it is sufficient to show the positive definiteness of $H_k+\frac1{\nu_k}\nabla c_k\nabla c_k^T$. For and $d\in R^n$, we can partition it into components in $\textrm{Null}(\nabla c_k^T)$ and $\textrm{Range}(\nabla c_k)$, and write
$$
d=t+\nabla c_kw
$$
with $t\in\textrm{Null}(\nabla c_k^T)$ and $w\in R^m$. Then we have that
\begin{equation}\label{c3}
\begin{split}
&d^T\left(H_k+\frac1\nu\nabla c_k\nabla c_k^T\right)d=(t+\nabla c_kw)^T\left(H_k+\frac1\nu\nabla c_k\nabla c_k^T\right)(t+\nabla c_kw)\\
=&t^T\left(H_k+\frac1\nu\nabla c_k\nabla c_k^T\right)t+2w^T\nabla c_k^T\left(H_k+\frac1\nu\nabla c_k\nabla c_k^T\right)t
+w^T\nabla c_k^T\left(H_k+\frac1\nu\nabla c_k\nabla c_k^T\right)\nabla c_kw\\
=&t^TH_kt+2w^T\nabla c_k^TH_kt+\frac2\nu w^T\nabla c_k\nabla c_k^T\nabla c_k^Tt+
w^T\nabla c_k^TH_k\nabla c_kw+\frac1\nu w^T\nabla c_k^T\nabla c_k\nabla c_k^T\nabla c_k w.
\end{split}\end{equation}
By similar arguments as the proof of \cite[Theorem 17.5]{WrighN1999}, there are non-negative scalars $a_1$, $a_2$, $a_3$, $a_4$
such that
\begin{eqnarray*}
&&t^TH_kt\geq a_1\|t\|^2,\ \textrm{for all } t\in\textrm{Null}\nabla c_k,\\
&&2w^T\nabla c_k^TH_kt\geq-a_2\|w\|\|t\|,\\
&&w^T\nabla c_k^TH_k\nabla c_kw\geq-a_3\|w\|^2,\\
&&\frac1\nu w^T\nabla c_k^T\nabla c_k\nabla c_k^T\nabla c_k w\geq\dfrac{a_4^2}\nu\|w\|^2
\end{eqnarray*}
By substituting these lower bounds into \eqref{c3}, we have
\begin{equation*}
\begin{split}
d^T\left(H_k+\frac1\nu\nabla c_k\nabla c_k^T\right)d\geq &a_1\|t\|^2-2a_2\|w\|\|t\|+\left(\dfrac{a_4^2}\nu-a_3\right)\|w\|^2\\
\geq&a_1\left[\|t\|-\frac{a_2}{a_1}\|w\|\right]^2+\left(\dfrac{a_4^2}{\nu}-c-\dfrac{a_2^2}{a_1}\right)\|w\|^2.
\end{split}\end{equation*}
Since $a_4>0$ by the full rank of $\nabla c_k$, $H_k+\frac1\nu\nabla c_k\nabla c_k^T$ is positive definite provided that we choose
$\bar\nu$ to be any value such that
$$
\dfrac{a_4^2}{\bar\nu}-c-\dfrac{a_2^2}{a_1}>0
$$
and choose $\nu\in(0,\bar\nu]$.
\end{proof}
However, if the condition of Theorem \ref{thmnu} is not satisfied, $\tilde W_k+\frac1{\nu_k}\nabla c_k\nabla c_k^T$ may be not
positive definite even if $\nu_k$ becomes extremely small. To prevent the algorithm from infinitely reducing $\nu_k$, we set a threshold $\nu_k^{\min}$ for $\nu_k$.
Define
$$
\nu_k^{\min}:=\begin{cases}\min\left\{\nu_{\min},\frac{\kappa_\nu\min\{\kappa_1(h_k^{\max}-\|c_k+\nabla c_k^Tv_k\|),\kappa_2(h_k-\|c_k+\nabla c_k^Tv_k\|)\}}{\min\left\{M_\nu,(\|\nabla\varphi_k^\mu+\tilde W_kv_k\|^2+1)\left(1+\frac{2\nu_0}{b_1}\|\tilde W_k+\zeta_k I\|\right)\right\}}\right\},&
\textrm{if } v_k>0,\\
\min\left\{\nu_{\min},\frac{\kappa_\nu\kappa_1h_k}{\min\left\{M_\nu,(\|\nabla\varphi_k^\mu+\tilde W_kv_k\|^2+1)\left(1+\frac{2\nu_0}{b_1}\|\tilde W_k+\zeta_k I\|\right)\right\}}\right\},&
\textrm{otherwise},
\end{cases}$$
where $\nu_{\min}\in(0,1)$ is close to 0, $\kappa_\nu$ is a positive scalar and $M_\nu$ is a large positive constant. If $\nu_k$ becomes smaller than $\nu_k^{\min}$ while
$\tilde W_k+\frac1{\nu_k}\nabla c_k\nabla c_k^T$ is still not positive definite, then set $\zeta_k$ to be a positive scalar such that
\begin{equation}\label{pdef}
d^T(\tilde W_k+\frac1{\nu_k}\nabla c_k\nabla c_k^T+\zeta_k I)d\geq b_1\|d\|^2
\end{equation}
for any $d\in R^n$ with $b_1>0$. Otherwise, just set $\zeta_k=0.$

Suppose that the current iterate is $x_k$, we describe algorithm for updating $\nu_k$ in Algorithm 2.

\framebox{\parbox{0.9\textwidth}{
A{\scriptsize LOGRITHM} 2: {U{\scriptsize PDATING} $\nu_k$}
\begin{enumerate}[\bf Step\ 1]\setcounter{enumi}{-1}
\item Set $\nu_k:=\nu_{k-1}$ and $\zeta_k=0$.

\item If $\tilde W_k+\frac1{\nu_k}\nabla c_k\nabla c_k^T+\zeta_k I$ is positive definite, go to Step 3.

\item If $\nu_k<\nu_k^{\min}$, set $\zeta_k$ to be a scalar satisfying \eqref{pdef}. Otherwise, go to Step 5.

\item Compute $t_k$.

\item If $t_k$ satisfies \eqref{b16} and \eqref{b17}, or if $t_k$ satisfies \eqref{b18} and \eqref{b19}, stop.

\item Set $\nu_k:=\nu_k/2$, go to Step 1.
\end{enumerate}
}}

~

Now we have the search direction $d_k=v_k+t_k$. The multiplier vector $\lambda$ corresponding to the next iterate is estimated by
\begin{equation}\label{elam}
\lambda_{k+1}=\frac1{\nu_k}\nabla c_k^Tt_k.
\end{equation}
The line search along $d_k$ is performed by first determining the maximal step length $\alpha_k^{\max}$ satisfying the
fraction-to-the-boundary rule
\begin{equation}\label{b12}
x_k+\alpha_k^{\max}d_k\geq(1-\tau_j)x_k,
\end{equation}  where $\tau_j\in(0,1)$ is a parameter close to 1 with respect to the iteration $j$ of the out loop.

If \eqref{b16} holds, then we call
the $k$th iteration an $f-$iteration and $x_k$ an $f-$iterate because it is quite reasonable to expect considerable reduction on objective function in
this case. Hence, we
require the step length $\alpha\in(0,\alpha_k^{\max}]$ to satisfy
 \begin{equation}\label{fred}
 \varphi^{\mu_j}_k-\varphi^{\mu_j}(x_k+\alpha d_k)\geq-\rho\alpha(\nabla \varphi^{{\mu_j}}_k)^Td_k.
 \end{equation}
The requirement for the feasibility on the new iterate is relatively rough. We require
 \begin{equation}\label{hmax}
 h(x_k+\alpha d_k)\leq h_k^{\max}.
 \end{equation}
 In the case where \eqref{b16} fails, which indicates that the infeasibility is significant while
sufficient reduction of objective function is not ensured, we call
the $k$th iteration an $h-$iteration and $x_k$ an $h-$iterate. We search $\alpha\in(0,\alpha_k^{\max}]$ satisfying
 \begin{equation}\label{hred}
 h(x_k+\alpha d_k)\leq (1-\rho)h_k+{\rho}\|c_k+\alpha\nabla c_k^Td_k\|
 \end{equation}
in this case.

 From \eqref{dz}, we obtain the estimation of the new dual variables
 \begin{equation}\label{ez}
 z_{k+1}={\mu_j} X_k^{-1}e-X_k^{-1}Z_kd_k.
 \end{equation}
 For the convergence proof, we require the ``primal-dual barrier term Hessian'' $X_k^{-1}Z_k$ do not deviate arbitrarily much from the
 ``primal Hessian'' $\mu_j X_k^{-2}$. To do this, we reset \cite{WaB2006}
 \begin{equation}\label{resetz}
 z_{k+1}^{(i)}:=\max\left\{\min\left\{z_{k+1}^{(i)},\frac{\kappa_{\sigma}\mu_j}{x_{k+1}^{(i)}}\right\},\frac{\mu_j}{\kappa_{\sigma}x_{k+1}^{(i)}}\right\}
 ,\ \ i=1,2,\cdots,n
 \end{equation}
 for some fixed $\kappa_\sigma>1$ after each step. Such safeguards not only benefit the convergence analysis but also work satisfcatory in practice \cite{ConnGOT2000,WaB2006,YamasYT2005}.

 After obtaining a new iterate, the limit on the new iterate is set as \cite{GouldT2010}
 \begin{equation}\label{hmaxnew}
 h_{k+1}^{\max}=\begin{cases}\begin{matrix}h_k^{\max}\, &\textrm{if } x_k \textrm{is an \it{f}-iterate},\\
 \max\{\kappa_hh_k^{\max},\bar\kappa_h h_k+(1-\bar\kappa_h) h_{k+1}\}, &\textrm{if } x_k \textrm{is an \it{h}-iterate}.\end{matrix}\end{cases}
 \end{equation}

 Now, we are ready to summarize all the details of this line search trust-funnel-like approach for the barrier problem. Suppose that the cuurent
 outer loop iteration is $j$ and that the parameter $\mu_j$, $\tau_j$ are available, and that the last iteration finished with the primal dual vector $(x_j,\lambda_j,z_j)$, where $(x_j,z_j)>0$. The detailed description is given in Algorithm 3.

\fbox{\parbox{0.9\textwidth}{
A{\scriptsize LGORITHM} 3: {I{\scriptsize NNER} L{\scriptsize OOP}}
\begin{enumerate}[\bf\ Step\ 1]\setcounter{enumi}{-1}
\item Choose $\delta\in(0,2]$, $\sigma_1,\sigma_2>0$, $\kappa_1,\kappa_2\in(0,1)$, $\rho\in 0,1$, $\kappa_\sigma>1$ and $\nu_0>0$.
Initialize the primal dual entry as
$(x_0,\lambda_0,z_0)=(x_j,\lambda_j,z_j).$ Let $$h_0^{\max}=\max\{h_0,\min(10,E_{\mu_j}(x_0,\lambda_0,z_0))\}.$$
Set $k:=0$.

\item If $E_{\mu_j}(x_k,\lambda_k,z_k)\leq\kappa_{\epsilon}\mu_j$, return.

\item Compute the normal step $v_k$ by solving \eqref{b9}. If $v_k=0$ and $h_k>0$, stop.

\item Use Algorithm 2 to update $\nu_k$ and compute the tangential step $t_k$.

\item Let $d_k=v_k+t_k$. Determine $\lambda_{k+1}$ by \eqref{elam}. Compute $z_{k+1}$ by \eqref{ez} and reset it by \eqref{resetz}.

\item Set $\alpha=\alpha_k^{\max}$ with $\alpha_k^{\max}$ defined by \eqref{b12}.

\item If \eqref{b16} holds, go to step 7. Otherwise, go to step 8.

\item \emph{f-iteration} \begin{enumerate}[Step\ 7.1]\setcounter{enumi}{0}
\item Set $x_k(\alpha)=x_k+\alpha d_k$.
\item If $x_k(\alpha)$ satisfies \eqref{fred} and \eqref{hmax}, go to Step 9.
\item Let  $\alpha=\alpha/2$ and go to step 7.1.
\end{enumerate}
\setcounter{enumi}{7}
\item \emph{h-iteration} \begin{enumerate}[Step\ 8.1]\setcounter{enumi}{0}
\item Set $x_k(\alpha)=x_k+\alpha d_k$.
\item If $x_k(\alpha)$ satisfies \eqref{hred}, go to Step 9.
\item Let  $\alpha=\alpha/2$ and go to step 8.1.
\end{enumerate}\setcounter{enumi}{8}

\item Set $\alpha_k=\alpha$, $x_{k+1}=x_k(\alpha_k)$. Compute $h_k^{\max}$ defined by \eqref{hmaxnew}.

 Set $k:=k+1$, go to step 1.
\end{enumerate}}}

~

\emph{Remark:} Algorithm 3 has chances to stop at Step 2, in which case the iterate $x_k$ fails to satisfy the linear
 independence constraint qualification and is a stationary point of $h(x)$,
meaning that $\nabla c_kc_k=0$.

\section{Well-definedness of Algorithm 3}

This section aims to give estimations on the parameter $\nu_k$ and the acceptable step length $\alpha_k$. The results established in this section not only
show the well-definedness of Algorithm 3 but also are essentially important for the global convergence analysis in the next section.
 We begin by recalling a result concerning a direct consequence of the
definition of $h_k^{\max}$.
\begin{lemma}\cite{GouldT2010}\label{lemhmax}The sequence $\{h_k^{\max}\}$ is non-increasing and the inequality
\begin{equation*}
0\leq h_l\leq h_k^{\max}
\end{equation*}
for all $l\geq k$.
\end{lemma}

Next we show that the parameter $\nu_k$ will admit the requirements if it becomes small enough.

\begin{lemma}
Suppose that $\nabla f_k$ and $\nabla c_k$ are Lipschitz continuous on an bounded open convex set $\Omega$ containing all the iterate generated by Algorithm 2,
that $\{B_k\}$ is bounded, and that the Algorithm 2 does not terminate at $x_k$. Then either a step $t_k$ satisfying \eqref{b16} and \eqref{b17}, or a step satisfying \eqref{b18} and \eqref{b19} will be found
if $\nu_k$ becomes small enough.
\end{lemma}
\begin{proof}
First, we give an estimate of the scale of $\|t_k\|$. By the optimality of $t_k$ for \eqref{b11}, we have that
\begin{eqnarray*}
\lefteqn{-\|\nabla\varphi_k^{\mu}+\tilde W_kv_k\|\|t_k\|+\frac12b_1\|t_k\|^2}\\
&&\leq (\nabla\varphi_k^{\mu}+\tilde W_kv_k)^Tt_k+\frac12t_k^T(W_k+\zeta_kI)t_k\\
&&\leq (\nabla\varphi_k^{\mu}+\tilde W_kv_k)^Tt_k+\frac12t_k^T\left(W_k+\frac1{\nu_k}\nabla c_k\nabla c_k^T+\zeta_kI\right)t_k\\
&&\leq 0,
\end{eqnarray*}
where the Cauchy-Schwartz inequality and \eqref{pdef} are used. It follows that
\begin{equation}\label{tup}
\|t_k\|\leq\frac{2}{b_1}\|\nabla\varphi_k^{\mu}+\tilde W_kv_k\|.
\end{equation}
By the first order necessary condition of the inexact tangential subproblem \eqref{b11}, the step $t_k$ satisfies
\begin{equation}\label{topt}
\nabla\varphi_k^{\mu}+\tilde W_kv_k+\left(W_k+\frac1{\nu_k}\nabla c_k\nabla c_k^T+\zeta_kI\right)t_k=0.
\end{equation}
Computing the inner products of $\nu_kt_k$ with both sides of this equation, moving terms not involving $\nabla c_k^Tt_k$
to the right hand side and using \eqref{tup} and Cauchy-Schwartz inequality, we obtain
\begin{equation}\label{tcon}
\|\nabla c_k^Tt_k\|^2\leq\frac{2\nu_k}{b_1}\|\nabla\varphi_k^{\mu}+\tilde W_kv_k\|^2\left(1+\frac{2\nu_k}{b_1}\|W_k+\zeta_kI\|\right).
\end{equation}

Note that if $v_k=0$ and the algorithm does not terminate, then $h_k=0$. Since Algorithm 3 does not stop at $x_k$, we have that $-(\nabla\varphi^\mu_k)^Tt_k>0$, which implies \eqref{b16}.
Then, from \eqref{tcon}, condition \eqref{b17} is satisfies if
\begin{equation}\label{numin1}
0<\nu_k\leq\frac{b_1\kappa_1h_k^{\max}}{2\|\nabla\varphi_k^\mu\|^2\left(1+\frac{2\nu_0}{b_1}\|W_k+\zeta_k I\|\right)}.
\end{equation}
If $v_k\neq0$, then from \eqref{tcon}, both \eqref{b17} and \eqref{b19} are satisfied if
\begin{equation}\label{numin}
0<\nu_k\leq\frac{b_1\min\{\kappa_1(h_k^{\max}-\|c_k+\nabla c_k^Tv_k\|),\kappa_2(h_k-\|c_k+\nabla c_k^Tv_k\|)\}}{2\|\nabla\varphi_k^\mu+\tilde W_kv_k\|^2\left(1+\frac{2\nu_0}{b_1}\|W_k+\zeta_k I\|\right)}.
\end{equation}
Since one of \eqref{b16} and \eqref{b18} always holds, the claim is true in this case.
\end{proof}

Obviously, the bounds in \eqref{numin1} and \eqref{numin} motivate the previously mentioned definition of $\nu_k^{\min}$. Next, we
use two lemmas to state the finite termination of Armijo line searches.

\begin{lemma}\label{lemfstep}
Suppose that $\nabla f_k$ and $\nabla c_k$ are Lipschitz continuous on a bounded open convex set $\Omega$ containing all the iterates generated by Algorithm 2,
that $\{B_k\}$ is bounded, and that the Algorithm 2 does not terminate at $x_k$. Suppose also that $t_k$ satisfies \eqref{b16} and \eqref{b17}. Then there
exists a positive scalar $\alpha_k^f$ such that for any $\alpha\in(0,\alpha_k^f]$, both \eqref{fred} and \eqref{hmax} are satisfied.
\end{lemma}
\begin{proof}
Firstly, we show that $-(\nabla\varphi_k^{\mu})^Td_k>0$. It is trivially true in the case $h_k>0$. Now we consider the
case where $h_k=0$. In this case, we have $v_k=0$ and $d_k=t_k$. If $t_k=0$, then by the first order necessary conditions of \eqref{b11} and \eqref{ez}, we have
$$
\nabla f_k+\nabla c_k\left(\frac1{\nu_k}\nabla c_k^Tt_k\right)-z_{k+1}=0.
$$
It follows that $x_k$ is a stationary point of \eqref{b1}, which contradicts with the assumptions of this lemma. Therefore, there must have $t_k\neq 0$. By \eqref{b11}
and \eqref{pdef},
we get
$$
-(\nabla\varphi_k^{\mu})^Td_k=-(\nabla\varphi_k^{\mu})^Tt_k\geq\frac12t_k^T(W_k+\zeta_kI)t_k\geq b_1\|t_k\|^2>0.
$$

Using the Lipschitz continuity of $\nabla f(x)$, we have for some Lipschitz constant $L_{df}$, such that for any step length $\alpha\in(0,\alpha_k^{\max}]$
\begin{eqnarray*}
\lefteqn{|f(x_k)-f(x_k+\alpha d_k)-(-\alpha\nabla f_k^Td_k)|}\\
&&\leq\alpha\displaystyle{\sup_{x_k^\tau\in[x_k,x_k+\alpha d_k]}}\|\nabla f(x_k^\tau)-\nabla f_k\|\|d_k\|\\
&&\leq \alpha^2 L_{df}\|d_k\|^2.
\end{eqnarray*}
Similarly, for any $i=1,\cdots,n$,
\begin{eqnarray*}
\lefteqn{\left|-\mu_j\sum_{i=1}^n\ln x_k^{(i)}+\mu_j\sum_{i=1}^n\ln(x_k^{(i)}+\alpha d_k^{(i)})-\mu_j\alpha\sum_{i=1}^n\frac{d_k^{(i)}}{x_k^{(i)}}\right|}\\
&&\leq\mu_j\sum_{i=1}^n\left|-\ln x_k^{(i)}+\sum_{i=1}^n\ln(x_k^{(i)}+\alpha d_k^{(i)})-\alpha\frac{d_k^{(i)}}{x_k^{(i)}}\right|\\
&&\leq\frac{\mu_j\alpha^2}{1-\tau_j}\sum_{i=1}^n\left(\frac{d_k^{(i)}}{x_k^{(i)}}\right)^2\leq\frac{\mu_j\alpha^2}{(1-\tau_j)\min_i(x_k^{(i)})^2}\|d_k\|^2.
\end{eqnarray*}

Using the abovementioned two inequalities, we have
\begin{equation*}
\begin{split}&|\varphi_k^{\mu}-\varphi^{\mu}(x_k+\alpha d_k)-(-\alpha(\nabla\varphi_k^{\mu})^Td_k)|\\
\leq&\alpha^2\left(L_{df}+\frac{\mu_j}{(1-\tau_j)\min_i(x_k^{(i)})^2}\right)\|d_k\|^2.\end{split}
\end{equation*}
This inequality implies
\begin{equation}\label{apd}
\begin{split}&\frac{|\varphi_k^{\mu}-\varphi^{\mu}(x_k+\alpha d_k)-(-\alpha(\nabla\varphi_k^{\mu})^Td_k)|}{-\alpha(\nabla\varphi_k^{\mu})^Td_k}\\
\leq&\alpha\frac{\left(L_{df}+\frac{\mu_j}{(1-\tau_j)\min_i(x_k^{(i)})^2}\right)\|d_k\|^2}{-(\nabla\varphi_k^{\mu})^Td_k}.\end{split}
\end{equation}
Let
$$
0<\alpha\leq\frac{-(1-\rho)(\nabla\varphi_k^{\mu})^Td_k}{\left(L_{df}+\frac{\mu_j}{(1-\tau_j)\min_i(x_k^{(i)})^2}\right)\|d_k\|^2}.
$$
Then it follows from \eqref{apd} that
\begin{equation*}
\frac{|\varphi_k^{\mu}-\varphi^{\mu}(x_k+\alpha d_k)-(-\alpha(\nabla\varphi_k^{\mu})^Td_k)|}{-\alpha(\nabla\varphi_k^{\mu})^Td_k}\leq 1-\rho,
\end{equation*}
which yields \eqref{fred}.

Next, we consider \eqref{hmax}.  Using \eqref{hnew}, \eqref{b17} and convexity, we have
\begin{equation}\label{c1}
\begin{split}
\|c(x_k+\alpha d_k)\|\leq&\|c_k+\alpha\nabla c_k^Td_k\|+\alpha^2 L_{dc}\|d_k\|^2\\
\leq&(1-\alpha)\|c_k\|+\alpha\|c_k+\nabla c_k^Tv_k\|+\alpha\|\nabla c_k^Tt_k\|+\alpha^2 L_{dc}\|d_k\|^2\\
\leq&(1-\alpha)h_k^{\max}+\alpha\|c_k+\nabla c_k^Tv_k\|+\alpha\kappa_1(h_k^{\max}-\|c_k+\nabla c_k^Tv_k\|)+\alpha^2 L_{dc}\|d_k\|^2\\
=&h_k^{\max}-\alpha(1-\kappa_1)(h_k^{\max}-\|c_k+\nabla c_k^Tv_k\|)+\alpha^2 L_{dc}\|d_k\|^2.
\end{split}\end{equation}
Since the algorithm does not terminate at $x_k$, we have $v_k\neq0$, which implies $h_k^{\max}-\|c_k+\nabla c_k^Tv_k\|>0$. Then, from \eqref{c1}, it follows that
\eqref{hmax} is satisfied if
$$
0<\alpha\leq\frac{(1-\kappa_1)(h_k^{\max}-\|c_k+\nabla c_k^Tv_k\|)}{L_{dc}\|d_k\|^2}.
$$

Summarizing all the arguments, both \eqref{fred} and \eqref{hmax} are satisfied  for any $\alpha\in(0,\alpha_k^f]$, where
\begin{equation}\label{alphaf}
\alpha_k^f:=\min\left\{\frac{-(1-\rho)(\nabla\varphi_k^{\mu})^Td_k}{\left(L_{df}+\frac{\mu_j}{(1-\tau_j)\min_i(x_k^{(i)})^2}\right)\|d_k\|^2},
\frac{(1-\kappa_1)(h_k^{\max}-\|c_k+\nabla c_k^Tv_k\|)}{L_{dc}\|d_k\|^2}\right\}.
\end{equation}
\end{proof}

\begin{lemma}\label{lemhstep}
Suppose that $\nabla f_k$ and $\nabla c_k$ are Lipschitz continuous on a bounded open convex set $\Omega$ containing all the iterates generated by Algorithm 2,
that $\{B_k\}$ is bounded, and that the Algorithm 2 does not terminate at $x_k$. Suppose also that $t_k$ satisfies \eqref{b18} and \eqref{b19}. Then there
exists a positive constant $\alpha_k^h$ such that for any $\alpha\in(0,\alpha_k^h]$, the condition \eqref{hred} is satisfied.
\end{lemma}
\begin{proof}
Since \eqref{b18} holds and the algorithm does not terminate at $x_k$, we have $v_k\neq 0$, which implies $h_k>0$. We firstly consider the case where $\nabla c_k$ is rank deficient. By the optimality of $v_k$ for \eqref{b9}, we have
\begin{equation*}
\|c_k+\nabla c_k^Tv_k\|^2\leq\|c_k+\nabla c_k^Tv_k\|^2+\|c_k\|^\delta\|v_k\|^2\leq\|c_k\|^2,
\end{equation*}
which yields
\begin{equation}\label{hpred}
\|c_k\|-\|c_k+\nabla c_k^Tv_k\|\geq\frac{\|c_k\|^\delta\|v_k\|^2}{\|c_k\|+\|c_k+\nabla c_k^Tv_k\|}\geq
\frac{\|c_k\|^\delta\|v_k\|^2}{2\|c_k\|}.
\end{equation}

By the Lipschitz continuity of $\nabla c(x)$, we have, for some Lipschitz constant $L_{dc}$ and any step length $\alpha\in(0,\alpha_k^{\max}]$, that
\begin{equation*}
\begin{split}
&\|c(x_k+\alpha d_k)\|-\|c_k+\alpha\nabla c_k^Td_k\|\\
&\leq\|c(x_k+\alpha d_k)-(c_k+\alpha\nabla c_k^Td_k)\|\\
&\leq \alpha^2\displaystyle{\sup_{x_k^{\tau}\in[x_k,x_k+\alpha d_k]}}\|\nabla c(x_k^\tau)-\nabla c_k\|\|d_k\|\\
&\leq\alpha^2 L_{dc}\|d_k\|^2,
\end{split}
\end{equation*}
which yields
\begin{equation}\label{hnew}
\|c(x_k+\alpha d_k)\|\leq\|c_k+\alpha\nabla c_k^Td_k\|+\alpha^2 L_{dc}\|d_k\|^2.
\end{equation}

Using \eqref{b19}, \eqref{hpred}, \eqref{hnew} and convexity, we have
\begin{equation*}
\begin{split}
\|c(x_k+\alpha d_k)\|\leq&\|c_k+\alpha\nabla c_k^Td_k\|+\alpha^2 L_{dc}\|d_k\|^2\\
=&(1-\rho)\|c_k+\alpha\nabla c_k^Td_k\|+\rho\|c_k+\alpha\nabla c_k^Td_k\|+\alpha^2 L_{dc}\|d_k\|^2\\
\leq&(1-\rho)((1-\alpha)\|c_k\|+\alpha\|c_k+\nabla c_k^Td_k\|)+\rho\|c_k+\alpha\nabla c_k^Td_k\|+\alpha^2 L_{dc}\|d_k\|^2\\
=&(1-\rho)\|c_k\|-(1-\rho)\alpha(\|c_k\|-\|c_k+\nabla c_k^Td_k\|)\\
&+\rho\|c_k+\alpha\nabla c_k^Td_k\|+\alpha^2 L_{dc}\|d_k\|^2\\
\leq&(1-\rho)\|c_k\|-(1-\rho)\alpha(\|c_k\|-\|c_k+\nabla c_k^Tv_k\|-\|\nabla c_k^Tt_k\|)\\
&+\rho\|c_k+\alpha\nabla c_k^Td_k\|+\alpha^2 L_{dc}\|d_k\|^2\\
\leq&(1-\rho)\|c_k\|-(1-\rho)(1-\kappa_2)\alpha(\|c_k\|-\|c_k+\nabla c_k^Tv_k\|)\\
&+\rho\|c_k+\alpha\nabla c_k^Td_k\|+\alpha^2 L_{dc}\|d_k\|^2\\
\leq&(1-\rho)\|c_k\|+\rho\|c_k+\alpha\nabla c_k^Td_k\|\\
&-\frac12(1-\rho)(1-\kappa_2)\alpha\|c_k\|^{\delta-1}\|v_k\|^2+\alpha^2 L_{dc}\|d_k\|^2
\end{split}
\end{equation*}
Define
$$
\alpha_k^h:=\frac{(1-\rho)(1-\kappa_2)\|c_k\|^{\delta-1}\|v_k\|^2}{2L_{dc}\|d_k\|^2}.
$$
 Then \eqref{hred} is satisfied for all $\alpha\in(0,\alpha_k^h]$.

 Next, we consider the other case where $\nabla c_k$ has full rank. In this case, we have
 $$
 v_k=-\nabla c_k(\nabla c_k^T\nabla c_k)^{-1}c_k,
 $$
 which satisfies
$$
c_k+\nabla c_k^Tv=0.
$$
Then $\|c_k\|-\|c_k+\nabla c_k^T v_k\|=\|c_k\|$. Using the similar arguments as the first case,  \eqref{hred} is satisfied for all $\alpha\in(0,\alpha_k^h]$ with
\begin{equation}\label{c5}
\alpha_k^h:=\frac{(1-\rho)(1-\kappa_2)\|c_k\|}{2L_{dc}\|d_k\|^2}.
\end{equation}
\end{proof}

Hence, we conclude that the presented algorithm is welldefineded.
\begin{theorem}
Suppose that $\nabla f_k$ and $\nabla c_k$ are Lipschitz continuous on a bounded open convex set $\Omega$ containing all the iterates generated by Algorithm 2,
that $\{B_k\}$ is bounded, then Algorithm 3 is welldefined.
\end{theorem}

\section{Global convergence analysis.}

To establish the global convergence theory for Algorithm 3, we need the following standard assumptions.
\begin{assumption}\label{ass}
Let $\{x_k\}$ be the sequence of the iterates generated by algorithm 3.
\begin{enumerate}[\ \ \ \ \ \ ({A}1)]
\item There is a bounded open convex set $\Omega$ containing $\{x_k\}$.

\item The Gradients $\nabla f_k$ and $\nabla c_k$ are Lipschitz continuous on $\Omega$.

\item The symmetric matrix $H_k$ is uniformly bounded.

\item The linear independence constraint qualification is satisfied at any  accumulation point $\tilde x$ of $\{x_k\}$.

 \item The (approximate) Hessian $H_k$ is uniformly positive definite on the null space of $\nabla c_k^T$, i.e., there is a neighborhood $\mathcal{N}$ of $\tilde x$ and a positive scalar $b>0$ such that
\begin{equation}\label{as5}
d^TH_kd\geq b\|d\|^2
\end{equation}
for all $x_k\in\mathcal{N}$ and $d\in\{p\ |\ \nabla c_k^Tp=0\}$.\end{enumerate}
\end{assumption}

%\emph{Remark:}

Assume that the algorithm does not terminate finitely. We begin our global convergence analysis by showing that, under Assumptions \ref{ass},
the components of the iterates generated by Algorithm 3 is bounded away
from 0 near the feasible region.
In \cite[Theorem 3]{WaB2005}, W{\"a}chter and Biegler pointed out that, under reasonable assumptions, for a give barrier parameter $\mu_j$, the
primal iterate sequence
$\{x_k\}$ generated by Newton-Lagrange methods for barrier problem is bounded away from 0. Here, we present the similar result in the context of
our algorithm and give a
more detailed proof, which follows the basic ideas of W{\"a}chter and Biegler's proof and uses the classical perturbation theory for linear system \cite{stewart1990matrix}. First, we consider the case where $x_k$ is not generated
by the restoration algorithm.
\begin{lemma}\label{lemd1}
Under Assumptions \ref{ass}, suppose that $\tilde x$ is an accumulation point of $\{x_k\}$, and that
 $$
 \mathcal{A}(\tilde x)=\{i\ |\ \tilde x^{(i)}=0\}
 $$
 is the set of active inequality constraints.
 Then there is a neighborhood $\mathcal{N}$ of $\tilde x$ such that for any $x_k\in\mathcal{N}$, there is
$d_k^{(i)}>0$ whenever $i\in\mathcal{A}(\tilde x)$.
\end{lemma}
\begin{proof}
For simplicity, we abbreviate $\mathcal{A}(\tilde x)$ as $\mathcal{A}$ in this proof.  Let $\bar{\mathcal{A}}$ be the set of inactive inequality constraints. Denote $n_a=|\mathcal{A}|$, $n_l=|\bar{\mathcal{A}}|$.
Without loss of generality, we assume $n_a\geq1$. Define $\tilde \delta=0.5\min\{\tilde x^{(i)}\ |\ i\in\bar{\mathcal{A}}\}$. Then there
is a neighborhood $\mathcal{N}$ such that $x_k^{(i)}\geq \tilde \delta$ for all $x_k\in\mathcal{N}$ and $i\in\bar{\mathcal{A}}$.

By Assumptions \ref{ass}, there is a (smaller) neighborhood $\mathcal{N}$ of $\tilde x$ such that $\nabla c_k$ has full rank and $H_k$ is
positive definite on the null
space of $\nabla c_k^T$ in the sense of \eqref{as5} for any $x_k\in\mathcal{N}$. Then from the definition of $v_k$, we have
\begin{equation}\label{vequ}
c_k+\nabla c_k^Tv_k=0.
\end{equation}
Let us shrink $\mathcal{N}$, if necessary, till $\nu_k^{\min}<\bar\nu$. By Lemma \ref{thmnu} and Algorithm 2,
$H_k+\frac1{\nu_k}\nabla c_k\nabla c_k^T$ will become positive definite before $\nu_k$ goes below $\nu_k^{\min}$,
which indicates that $\zeta_k=0$. It follows from \eqref{topt} that
\begin{equation}\label{tequ}
(H_k+X_k^{-1}Z_k)^Td_k+\nabla c_k\left(\frac1{\nu_k}\nabla c_k^Tt_k\right)=-\nabla f_k+\mu_jX_k^{-1}e.
\end{equation}
From \eqref{b17}, \eqref{b19}, \eqref{elam}, \eqref{vequ} and \eqref{tequ}, we have the following linear system
\begin{equation}\label{iplin}
\begin{pmatrix}
H_k+X_k^{-1}Z_k&\nabla c_k\\
\nabla c_k^T&0
\end{pmatrix}\begin{pmatrix}d_k\\ \lambda_{k+1}\end{pmatrix}=
-\begin{pmatrix}
\nabla f_k-\mu_jX_k^{-1}e\\
c_k-\nabla c_k^Tt_k.
\end{pmatrix}
\end{equation}

Following the ideas of W{\"a}chter and Biegler but with slight difference, we denote with $x_k^a$ the vector of
components in $\mathcal{A}$ and with $x_k^l$ the vector of remaining ones.
Without loss of generality, we assume that $x_k=(x_k^a,x_k^l)$. Similarly, we define $\nabla c_k^a$, $\nabla c_k^l$ and etc.
Rewrite the linear system \eqref{iplin} by scaling the first rows and columns
by $X_k^a$:
\begin{equation}\label{d3}
\begin{pmatrix}
X_k^aH_k^{aa}+X_k^aZ_k^a&X_k^aH_k^{al}&X_k^a\nabla c_k\\
H_k^{la}X_k^a\nabla c_k&H_k^{ll}+(X_k^{l})^{-1}Z_k^{l}&\nabla c_k^{l}\\
(\nabla c_k^a)^TX_k^a&(\nabla c_k^l)^T&0
\end{pmatrix}\begin{pmatrix}\tilde d_k^a\\ d_k^l\\ \lambda_{k+1}\end{pmatrix}=
-\begin{pmatrix}
X_k^a\nabla f_k^a-\mu_je_a\\
\nabla f_k^l-\mu_j(X_k^l)^{-1}e_l\\
c_k-\nabla c_k^Tt_k
\end{pmatrix},
\end{equation}
where $\tilde d_k^a=(X_k^a)^{-1}d_k^a$, $e_a\in R^{n_a}$ and $e_l\in R^{n_l}$.
Partitioning the coefficient matrix of this linear system into
$$
\begin{pmatrix}W_{11}&W_{12}\\ W_{21}& W_{22}\end{pmatrix},
$$
where
$$\begin{array}{cc}
W_{11}=X_k^aH_k^{aa}+X_k^aZ_k^a,& W_{12}=(X_k^aH_k^{al}\ X_k^a\nabla c_k )\\
W_{21}=\begin{pmatrix}H_k^{la}X_k^a\nabla c_k\\ (\nabla c_k^a)^TX_k^a
\end{pmatrix},& W_{22}=\begin{pmatrix}H_k^{ll}+(X_k^{l})^{-1}Z_k^{l}&\nabla c_k^{l}\\
(\nabla c_k^l)^T&0\end{pmatrix}.
\end{array}$$
By assumption (A4), the matrix
$$
\begin{pmatrix}
\nabla c^s(\tilde x)&I\\
\nabla c^l(\tilde x)&0
\end{pmatrix}
$$
has full rank, from which it follows that $\nabla c^l(\tilde x)$ has full rank. Then there is (a smaller) neighborhood
$\mathcal{N}$ of
$\tilde x$ such that $\nabla c_k^l$ has full rank. For any $d^l$ in $\textrm{Null}((\nabla c_k^l)^T)$, the vector
$\bar d=(0,d^l)$ lies in $\textrm{Null}(\nabla c_k^T)$. Thus, from assumption (A5), we get
$$
(d_k^l)^TH_k^{ll}d_k^l=\bar d^T H_k\bar d\geq b\|\bar d\|^2=b\|d^l\|^2,
$$
i.e., $H_k^{ll}$ is positive definite in $\textrm{Null}((\nabla c_k^l)^T)$. Then $W_{22}$ is nonsingular in $\mathcal{N}$.
By Assumptions \ref{ass} and using $x_k^l\geq\tilde \delta e_l$ and \eqref{resetz}, their is a positive constant $M_W$ such that
$\|W_{22}^{-1}\|\leq M_W$. By eliminating $d_k^l$ and $\lambda_{k+1}$ in \eqref{d3}, we obtain
\begin{equation}\label{dx1}
(W_{11}-W_{12}W_{22}^{-1}W_{21})\tilde d_k^a=-(X_k^a\nabla f_k^a-\mu_je_a-W_{12}W_{22}^{-1}g),
\end{equation}
where
$$
g=\begin{pmatrix}\nabla f_k^l-\mu_j(X_k^l)^{-1}e_l\\
c_k-\nabla c_k^Tt_k
\end{pmatrix}.
$$
Using $x_k^l\geq\tilde \delta e_l$, there is $M_g>0$ such that $\|g\|\leq M_g$ for all $x_k\in\mathcal{N}$.
Consider the linear system
\begin{equation}\label{dx2}
X_k^aZ_k^a\bar d_k^s=\mu_je_a.
\end{equation}
System \eqref{dx1} can be viewed as a perturbed system for \eqref{dx2} with the coefficient matrix perturbed by
$$
G=X_k^aH_k^{aa}-W_{12}W_{22}^{-1}W_{21}
$$
and the right hand side by
$$
r=-(X_k^a\nabla f_k^a-W_{12}W_{22}^{-1}g).
$$
By perturbation theory for linear system, see, for instance, \cite{stewart1990matrix}, we have
\begin{equation}\label{pls}
\dfrac{\|\tilde d_k^a-\bar d_k^a\|_{\infty}}{\|\bar d_k^a\|_{\infty}}\leq\dfrac{\kappa(X_k^aZ_k^a)}
{1-\kappa(X_k^aZ_k^a)\dfrac{\|G\|_{\infty}}{\|X_k^aZ_k^a\|_{\infty}}}
\left(\dfrac{\|G\|_{\infty}}{\|X_k^aZ_k^a\|_{\infty}}+\frac{\|r\|_{\infty}}{\mu_j}\right),
\end{equation}
where $\kappa(X_k^aZ_k^a)$ refers to the condition number of $X_k^aZ_k^a$, provided that
\begin{equation*}
\|(X_k^aZ_k^a)^{-1}\|_{\infty}\|G\|_{\infty}<1.
\end{equation*}
Using \eqref{resetz}, we have that
\begin{equation}\label{dx3}
\dfrac{\mu_j}{\kappa_\Sigma}I\leq X_k^aZ_k^a\leq\kappa_{\Sigma}\mu_jI,
\end{equation}
which implies that
 \begin{equation}\label{dx5}\dfrac{\mu_j}{\kappa_\Sigma}\leq\|X_k^aZ_k^a\|_{\infty}\leq\kappa_{\Sigma}\mu_j,\  \|(X_k^aZ_k^a)^{-1}\|_{\infty}\leq\kappa_\Sigma/\mu_j,\end{equation} and that the solution $\bar d_k^s$ of \eqref{dx2} satisfies
\begin{equation}\label{d1}
\dfrac{1}{\kappa_\Sigma}\leq\bar d_k^{(i)}\leq\kappa_\Sigma,\ \forall i\in \mathcal{A}.
\end{equation}
Denote
$$
\delta_\mu=\max\{x_k^{(i)}\ |\ i\in\mathcal{A}\}.
$$
From Assumptions \ref{ass} and using $x_k^l\geq\tilde\delta e_l$, there are positive constants $M_H$, $M_{12}$, $M_{21}$, $M_{df}$ and $M_g$ such that
$$
\|G\|_{\infty}\leq\|X_k^a\|_{\infty}\|H_k^{aa}\|_{\infty}+\|W_{12}\|_{\infty}\|W_{22}^{-1}\|_{\infty}\|W_{21}\|_{\infty}\leq M_H\delta_\mu+M_{12}M_WM_{21}\delta_\mu^2
$$
and
$$
\|r\|_{\infty}\leq M_{df}\delta_\mu+ M_{12}M_WM_g\delta_\mu
$$
From \eqref{dx5} and using $\delta_\mu\rightarrow 0$ as $x_k\rightarrow\tilde x$, there is a (smaller) neighborhood $\mathcal{N}$ of $\tilde x$ such that
$$
1-\kappa(X_k^aZ_k^a)\dfrac{\|G\|_{\infty}}{\|X_k^aZ_k^a\|_{\infty}}=1-\|(X_k^aZ_k^a)^{-1}\|\|G\|_{\infty}\geq\frac12.
$$
It also follows from \eqref{dx5} that $\kappa(X_k^aZ_k^a)\leq\kappa_\Sigma^2$. Having all these facts in mind, the right hand side of the inequality
\eqref{pls} tends to 0 as $x_k$ approaching $\tilde x$.
It follow that \begin{equation*}
\begin{split}
\tilde d_k^{(i)}\geq&\bar d_k^{(i)}-|\tilde d_k^{(i)}-\bar d_k^{(i)}|\geq\bar d_k^{(i)}-\|\tilde d_k^a-\bar d_k^a\|_{\infty}\\
\geq&\bar d_k^{(i)}-\|\tilde d_k^a-\bar d_k^a\|_{\infty}\|\bar d_k^a\|_{\infty}\\
\geq&\bar d_k^{(i)}-\dfrac{\kappa(X_k^aZ_k^a)}
{1-\kappa(X_k^aZ_k^a)\dfrac{\|G\|_{\infty}}{\|X_k^aZ_k^a\|_{\infty}}}
\left(\dfrac{\|G\|_{\infty}}{\|X_k^aZ_k^a\|_{\infty}}+\frac{\|r\|_{\infty}}{\mu_j}\right)\kappa_\Sigma\\
\geq&\frac{1}{\kappa_\Sigma}-\dfrac{\kappa(X_k^aZ_k^a)}
{1-\kappa(X_k^aZ_k^a)\dfrac{\|G\|_{\infty}}{\|X_k^aZ_k^a\|_{\infty}}}
\left(\dfrac{\|G\|_{\infty}}{\|X_k^aZ_k^a\|_{\infty}}+\frac{\|r\|_{\infty}}{\mu_j}\right)\kappa_\Sigma\\
\rightarrow&\frac{1}{\kappa_\Sigma}>0,
\end{split}\end{equation*}
where \eqref{pls}, \eqref{d1} are used. Therefore, by shrinking $\mathcal{N}$ if necessary, we have
$
\tilde d_k^a> 0
$
for all $x_k\in\mathcal{N}$, which yields $d_k^a> 0$ because $d_k^a=X_k^a\tilde d_k^a$.
\end{proof}

\begin{theorem}\label{thmd3}
Under Assumptions \ref{ass}, there is a positive constant $\tilde \epsilon>0$, such that for any $x_k\geq\tilde\epsilon e$ for any index $k$.
\end{theorem}
\begin{proof}
 By Lemma \ref{lemd1}, for any accumulation point $\tilde x$ of $\{x_k\}$, there is a neighborhood $\tilde{\mathcal{N}}$ of it, such that
 $d_k^{(i)}>0$, $i\in\mathcal{A}$ for all $x_k\in\tilde{\mathcal{N}}$. By compactness, there are finite number of such neighborhoods, denoting with $\mathcal{N}_1$,$\cdots$, $\mathcal{N}_q$, of differen accumulation points, denoting with $\tilde x_1$,$\cdots$, $\tilde x_q$, whose
union covers all the accumulation points. Define $\tilde\delta_p=0.5\min\{\tilde x_p^{(i)}\ |\ i\not\in\mathcal{A}(\tilde x_p) \}$,
$p=1,\cdots,q$. Then $x_k^{(i)}\geq\tilde \delta_p$, for any $i\not\in \mathcal{A}(\tilde x_p)$ and $x_k\in\mathcal{N}_p$. Denote
$\mathcal{N}=\bigcup_{p=1}^q\mathcal{N}_p$. Then there are only finite iterates that are not covered by $\mathcal{N}$.
Let
$$
\tilde\epsilon:=(1-\tau_j)\min\{\min_{p=1,\cdots,q}\tilde\delta_p,\min_{x_k\not\in\mathcal{N},i=1,\cdots,n}x_k^{(i)}\}.
$$
Suppose, without loss of generality, that $x_0\geq\tilde\epsilon e$. Assume that $x_k\geq\tilde\epsilon e$. Let us consider
$x_{k+1}$. If $x_k\not\in\mathcal{N}$, then the lemma is trivially true. Now we consider the case where $x_k\in\mathcal{N}$.
Note that $x_k^{(i+1)}<x_k^{(i)}$ occurs only if $i$ is not active constraints. Then we have that
$$
x_k^{(i+1)}\geq\begin{cases} x_k^{(i)}\geq\tilde\epsilon,& \textrm{the index }i \textrm{ is active},\\
(1-\tau)x_k^{(i)}\geq(1-\tau_j)\min_{p=1,\cdots,q}\tilde\delta_p\geq\tilde\epsilon, &\textrm{otherwise}.
\end{cases}
$$

\end{proof}

From Assumptions \ref{ass}, Theorem \ref{thmd3} and \eqref{resetz}, it is easy to conclude the boundedness of search direction $\{d_k\}$.
From now on, we denote with $M_d$ the upper bound of $\{\|d_k\|\}$. The boundedness of the Lagrange multipliers $\{\lambda_k\}$ follows from
\eqref{topt}.

The remainder of this section gives global convergence results. We shall use the following two index sets
$$
\mathcal{K}_h=\{k\ |\ h_{k+1}^{\max}< h_k^{\max}\},\bar{\mathcal{K}}_h=\{k\ |\ h_{k+1}^{\max}=\kappa_h h_{k}^{\max}\},
$$
which containing $h-$iterations. First, we consider the case where $\mathcal{K}_h$ is infinite.
\begin{lemma}\label{lemfes1}
Suppose that $|\mathcal{K}_h|=+\infty$. Then $\lim_{k\rightarrow\infty}h_k=0$.
\end{lemma}
\begin{proof}
Without loss of generality, we assume that $\nabla c_k$ has full rank for all $k\geq 0$. By the updating rule of $h_k^{\max}$ \eqref{hmaxnew}, we have
\begin{eqnarray*}
\lefteqn{h_k^{\max}-h_{k+1}^{\max}\geq\min\{(1-\kappa_h)h_k^{\max},(1-\bar\kappa_h)(h_k-h_{k+1})\}}\\
&&\geq\min\{(1-\kappa_h)h_k,(1-\kappa_2)\rho\alpha_kh_k\},
\end{eqnarray*}
where Lemma \ref{lemhmax}, \eqref{b9}, \eqref{b19} and \eqref{hred} are used. Since $\alpha_k\geq0.5\alpha_k^h$, where $\alpha_k^h$ given by \eqref{c5},
the previous inequalities give
\begin{equation*}
{h_k^{\max}-h_{k+1}^{\max}\geq\min\left\{(1-\kappa_h)h_k,\dfrac{(1-\kappa_2)^2(1-\rho)\rho h_k}{2L_{dc}M_d^2}\right\}},
\end{equation*}
where the upper bound of $\{\|d_k\|\}$ are used. From the non-increasing property of the sequence $\{h_k^{\max}\}$ (see Lemma \ref{lemhmax}), it follows from
the above inequality that \begin{equation}\label{dx9}\lim_{k\in\mathcal{K}_h}h_k=0.\end{equation}

Consider the cardinal number of $\bar{\mathcal{K}}_h$. If $|\bar{\mathcal{K}}_h|$, then it is a direct consequence of the definition of $\bar{\mathcal{K}}_h$ that $\lim_{k\rightarrow\infty}h_k^{\max}=0$. Otherwise, there is an index $k_1>0$ such that
$$
h_k^{\max}=\kappa_h h_k+(1-\kappa_h)h_{k+1}\leq h_k
$$
for all $k\geq k_1$ and $k\in\mathcal{K}_h$. Then we have by this inequality and \eqref{dx9}
$$
\lim_{k\in\mathcal{K}_h}h_k^{\max}=0.
$$
Using non-increasing property of $\{h_k^{\max}\}$, we have $\lim_{k\rightarrow\infty}h_k^{\max}=0$, which implies, by Lemma \ref{lemhmax}, that
$\lim_{k\rightarrow\infty}h_k=0$.
\end{proof}

\begin{lemma}\label{lemt1}
Suppose that $|\mathcal{K}_h|=+\infty$. Then $\lim_{k\in\mathcal{K}_h}t_k=0$.
\end{lemma}
\begin{proof}
Since $\lim_{k\rightarrow\mathcal{K}_h}h_k=0$ and \eqref{b18} holds for all $k\in\mathcal{K}_h$, we have
\begin{equation*}
\|(\nabla\varphi_k^{\mu_j})^Tt_k\|-\|(\nabla\varphi_k^{\mu_j})^Tv_k\|\leq\sigma_1h_k^{\sigma_2},
\end{equation*}
which implies
$\lim_{k\in\mathcal{K}_h}\|(\nabla\varphi_k^{\mu_j})^Tt_k\|=0.$ By the optimality of $t_k$ for \eqref{b11} and using \eqref{pdef}, we have
$$
-(\|(\nabla\varphi_k^{\mu_j})^Tt_k\|+\|W_kt_k\|\|v_k\|)+\frac12b_1\|t_k\|^2\leq 0.
$$
It follows that
\begin{equation}\label{dx8}
\frac12b_1\|t_k\|^2\leq(\|(\nabla\varphi_k^{\mu_j})^Tt_k\|+\|W_kt_k\|\|v_k\|).
\end{equation}
Taking limits on both sides, we get $\lim_{k\in\mathcal{K}_h}t_k=0$.
\end{proof}

Next, we consider the case where $|\mathcal{K}_h|$ is a finite set.

\begin{lemma}\label{lemfes2}
Suppose that $|\mathcal{K}_h|<+\infty$. Then $\lim_{k\rightarrow+\infty}h_k=0$.
\end{lemma}
\begin{proof}
Finiteness of the set $\mathcal{K}_h$ implies, using the updating rule of $h_k^{\max}$, that there is an index $k_0>0$ such that for all $k\geq k_0$,
$x_k$ is an $f$-iterate, i.e., the inequality \eqref{b16} holds. Then following the acceptance rule for $f-$steps, we have
$$
\varphi_k^{\mu_j}-\varphi_{k+1}^{\mu_j}\geq-\rho\alpha_k(\nabla_k^{\mu_j})^Td_k.
$$
Using this inequality and \eqref{b16} gives
\begin{equation}\label{dx6}
\varphi_k^{\mu_j}-\varphi_{k+1}^{\mu_j}\geq\rho\sigma_1\alpha_kh_k^{\sigma_2}.
\end{equation}
Without loss of generality, we still assume that $\nabla c_k$ has full rank for all $k\geq 0$. By Lemma \eqref{lemfstep}, we have $\alpha_k\geq0.5\alpha_k^f$, where $\alpha_k^f$ is given in \eqref{alphaf}. Noting that
$c_k+\nabla c_kv_k=0$ and the algorithm does not update $h_k^{\max}$ for $k\geq k_0$, \eqref{dx6} yields
\begin{equation}%\label{dx7}
\varphi_k^{\mu_j}-\varphi_{k+1}^{\mu_j}\geq\frac12\rho\sigma_1\min\left\{\frac{(1-\rho)\sigma_1h_k^{\sigma_2}}
{\left(L_{df}+\frac{\mu_j}{(1-\tau_j)\tilde\epsilon^2}\right)M_d^2},
\frac{(1-\kappa_1)(h_{k_0}^{\max})}{L_{dc}M_d^2}\right\}h_k^{\sigma_2}
\end{equation}
for all $k\geq k_0$, where Theorem \ref{thmd3}, \eqref{b16} and the upper bound of $\{\|d_k\|\}$ are used. Note that Theorem \ref{thmd3} ensures the boundedness of $\{\varphi_k^{\mu_j}\}$ and the acceptance criteria for $f-$steps implies the non-increasing property for $\{\{\varphi_k^{\mu_j}\}\}_{k\geq k_0}$. It
follows from \eqref{dx6} that $\lim_{k\rightarrow\infty}h_k=0$.
\end{proof}

\begin{lemma}\label{lemt2}
Suppose that $|\mathcal{K}_h|<+\infty$. Then $\lim_{k\rightarrow+\infty}t_k=0$.
\end{lemma}
\begin{proof}
By similar arguments as the proof of the previous lemma, we have
\begin{equation*}
\varphi_k^{\mu_j}-\varphi_{k+1}^{\mu_j}\geq-\frac12\rho\sigma_1\min\left\{\frac{-(1-\rho)(\nabla\varphi_k^{\mu_j})^Td_k}
{\left(L_{df}+\frac{\mu_j}{(1-\tau_j)\tilde\epsilon^2}\right)M_d^2},
\frac{(1-\kappa_1)(h_{k_0}^{\max})}{L_{dc}M_d^2}\right\}h_k^{\sigma_2}(\nabla\varphi_k^{\mu_j})^Td_k,
\end{equation*}
for all $k\geq k_0$. It follows that $\lim_{k\rightarrow\infty}\|(\nabla\varphi_k^{\mu_j})^Td_k\|=0$, where the non-increasing property and lower
boundedness of $\{\varphi_k^{\mu_j}\}$ are used. Since $\lim_{k\rightarrow\infty}h_k=0$, we get $\lim_{k\rightarrow\infty}\|(\nabla\varphi_k^{\mu_j})^Tt_k\|=0$. Therefore, by \eqref{dx8}, we have $\lim_{k\rightarrow\infty}t_k=0.$
\end{proof}

To sum up the above four lemmas, i.e., Lemmas \ref{lemfes1}-\ref{lemt2}, we get our global convergence theorem.

\begin{theorem}\label{thminner}
Under Assumptions 4.1, suppose that Algorithm 3 does not terminate finitely then
\begin{enumerate}[(1)]
\item if $|\mathcal{K}_h|=\infty$, then any accumulation point of $\{x_k\}_{k\in{\mathcal{K}}_h}$ is a KKT point for \eqref{b1}.

\item if $|\mathcal{K}_h|<\infty$, then any accumulation point of $\{x_k\}$ is a KKT point for \eqref{b1}.
\end{enumerate}
\end{theorem}

Theorem \ref{thminner} indicates that the inner loop of Algorithm 1, i.e., Step 2 will terminate finitely under Assumptions 4.1. Hence, by the mechanism of the
algorithm, we get the global convergence of the whole algorithm.

\begin{theorem}
Under Assumptions 4.1, if Algorithm 1 does not terminate finitely, then at least one of the accumulation points of the iterate sequence is a KKT point for
problem \eqref{a1}.
\end{theorem}

%\bibliographystyle{plain}
%\bibliography{MyRef}

\end{document}